\documentclass[12pt]{article}
\usepackage{fullpage}
\usepackage{amssymb}
\usepackage{amsmath,amsthm,amssymb}

\newtheorem{defn}{Definition}[section]
\newtheorem{thm}[defn]{Theorem}
\newtheorem{lemma}[defn]{Lemma}
\newtheorem{coro}[defn]{Corollary}
\newtheorem{clm}[defn]{Claim}
\theoremstyle{definition}

\newtheorem*{rmrk}{Remark}
\newtheorem*{teo}{Theorem}

\title{On Galvin's lemma and Ramsey spaces}

\author{Jos\'e G. Mijares\thanks{jmijares@ivic.ve}\ \thanks{jose.mijares@ciens.ucv.ve}\\ Departamento de Matem\'aticas\\ Instituto Venezolano de Investigaciones Cient\'ificas y \\Escuela de Matem\'aticas\\Universidad Central de Venezuela}
\date{}

\begin{document}

\maketitle

\begin{abstract}
An abstract version of Galvin's lemma is proven, within the framework of the theory of Ramsey spaces. Some instances of it are explored.
\end{abstract}

\begin{flushleft}
\end{flushleft}

\section{Introduction}

 For $A\subseteq\mathbb{N}$, let $A^{[\infty]} =
\{X\subset A: |X| = \infty\}$ and $A^{[<\infty]} =
\{X\subset A: X\ \mbox{is finite}\ \}$. Galvin's lemma can be stated as follows:

\begin{teo}[Galvin's lemma \cite{galvin}]
 Given ${\cal F}\subseteq\mathbb{N}^{[<\infty]}$, there exists $A\in \mathbb{N}^{[\infty]}$ such that one of the following holds:
\begin{enumerate}
\item $A^{[<\infty]}\cap{\cal F}=\emptyset$, or
\item $(\forall B\in A^{[\infty]})\ (\exists \ a\in {\cal F})\ (a\sqsubset B)$, i.e., $a$ is an \textit{initial segment} of $B$.
\end{enumerate}
\end{teo}

 This important result plays a crucial role in the characterization of those subsets of $\mathbb{N}^{[\infty]}$ having the \textit{Ramsey property}. It deals with finite colorings of the set of natural \textit{approximations} to infinite sets of nonnegative integers (i.e., finite subsets of them) and  makes possible to show that some interesting subsets of $\mathbb{N}^{[\infty]}$ are Ramsey. This was the approach used by Galvin and Prikry (see \cite{galpri}) to show that metric Borel subsets of $\mathbb{N}^{[\infty]}$ are Ramsey. After Ellentuck gave (in \cite{ellen}) a topological characterization of the Ramsey property, several Ellentuck-like theorems which generalize this characterization to other contexts were proven (see for instance \cite{carlson}, \cite{carsimp},
\cite{mill} or \cite{todo}). Each of these theorems deals with a \textit{topological Ramsey space}, endowed with a convenient set of approximations to its elements and with a topology similar to the one defined by Ellentuck on $\mathbb{N}^{[\infty]}$ (in \cite{todo}, these results are condensed into the \textit{abstract Ellentuck theorem}, from which all of them can be derived). Nevertheless, given one such Ramsey space, the nature of the set of approximations related to it in a sense expressed by Ramsey's theorem \cite{ramsey} and Galvin's lemma, is explored using an \textit{indirect approach} in most of the cases. That is, given a topological Ramsey space, the statements about the regular behavior of the corresponding set of approximations are derived from those concerning the regular behavior of subsets of the space, using the corresponding Ellentuck-like theorem.

\medskip

 Following \cite{carsimp} and \cite{todo}, but avoiding to use the abstract Ellentuck theorem, in this work we show an abstract version of Galvin's lemma, within the framework of the theory of Ramsey spaces. Any instance of it is a true combinatorial statement concerning the regular behavior of the corresponding set of approximations in a given topological Ramsey space. Among the many instances, we present one which lead us to a simple proof of the Graham-Leeb-Rothschild theorem \cite{graleeroth}, which refers to finite colorings of finite dimensional vector spaces over a finite field, and of an infinitary version of it due to Carlson \cite{carlson} which can be seen as a vector version of the Galvin-Prikry theorem \cite{galpri}. In the same spirit, we present another instance leading to simple proofs of Ramsey's theorem for $n$-parameter sets due to Graham and Rothschlid \cite{graroth}, of the dualization of Ramsey's theorem due to Halbeisen \cite{Halb} and of the Dual Galvin-Prikry theorem due to Carlson and Simpson \cite{carsimp2}.

\begin{flushleft}
\textbf{Acknowledgement.} \\
The author would like to express his deepest gratitude to Stevo Todorcevic and Carlos Di Prisco. Also, the author thanks Labib S. Haddad for valuable comments that helped make the proof of our main result more readable, and thanks the referee for useful suggestions to make a better presentation of the results contained in this paper.
\end{flushleft}

\section{Topological Ramsey spaces}\label{topo}
 The definitions and results throughout this section are
taken from \cite{todo}. A previous presentation can also be found in \cite{carsimp}. Consider a triplet of the form $(\mathcal{R}, \leq,
r )$, where $\mathcal{R}$ is a set, $\leq$
is a quasi order on $\mathcal{R}$ and $r: \mathbb{N}\times\mathcal{R}\rightarrow \mathcal{AR}$
 is a function with range $\mathcal{AR}$. For every $n\in \mathbb{N}$ and every $A\in \mathcal{R}$, let us write $r_n(A) := r(n,A)$ and $\mathcal{AR}_n := \{r_n(A) : A\in\mathcal{R}\}$. We say that $r_{n}(A)$ is
 \textit{the} $n$th \textit{approximation of} $A$. In order to capture the combinatorial structure
 required to ensure the provability of an Ellentuck type theorem, some assumptions on
 $(\mathcal{R}, \leq, r)$ will be imposed. The first three of them are the following:

\begin{itemize}
\item[{(A.1)}]For any $A\in \mathcal{R}$, $r_{0}(A) = \emptyset$.
\item[{(A.2)}]For any $A,B\in \mathcal{R}$, if $A\neq B$ then
$(\exists n)\ (r_{n}(A)\neq r_{n}(B))$.
\item[{(A.3)}]If $r_{n}(A) =
r_{m}(B)$ then $n = m$ and $(\forall i<n)\ (r_{i}(A) = r_{i}(B))$.
\end{itemize}
These three assumptions allow us to identify each $A\in
\mathcal{R}$ with the sequence $(r_{n}(A))_{n}$ of its
approximations. In this way, if $\mathcal{AR}$ has the discrete
topology, $\mathcal{R}$ can be identified with a subspace of the
(metric) space $\mathcal{AR}^{\mathbb{N}}$ (with the product
topology) of all the sequences of elements of $\mathcal{AR}$. We will say that $\mathcal{R}$ is
{\it metrically closed} if it is a closed subspace of
$\mathcal{AR}^{\mathbb{N}}$. The basic open sets generating the metric topologogy on $\mathcal{R}$ inherited from the product topology of $\mathcal{AR}^{\mathbb{N}}$ are of the form:

$$[a] = \{B\in \mathcal{R} : (\exists n)(a =
r_{n}(B))\}$$where $a\in\mathcal{AR}$.

\medskip

 For $a\in \mathcal{AR}$, define the \textit{length} of $a$,
$|a|$, as \textit{the unique} $n$ \textit{such that} $a = r_{n}(A)$ \textit{for some} $A\in
\mathcal{R}$. The \textit{Ellentuck type
neighborhoods} are of the form:
$$[a,A] = \{B\in \mathcal{R} : (\exists n)(a =
r_{n}(B))\ \ \mbox{and} \ \ (B\leq A)\}$$ where $a\in\mathcal{AR}$
and $A\in \mathcal{R}$. Let $\mathcal{AR}(A) = \{a\in\mathcal{AR} : [a,A]\neq\emptyset\}$. Also, write $[n,A] := [r_{n}(A),A]$.

\medskip

 Also, given a neighborhood $[a,A]$ and $n\geq |a|$, let $r_n[a,A]$ be \textit{the image of} $[a,A]$ \textit{by the function} $r_n$, i.e., the set  $\{b\in\mathcal{AR} : \exists B\in [a,A]\ \mbox{such that}\ b = r_n(B)\}$.

\begin{defn}
A set $\mathcal{X}\subseteq \mathcal{R}$
is \textbf{Ramsey} if for every neighborhood $[a,A]\neq\emptyset$
there exists  $B\in [a,A]$ such that $[a,B]\subseteq \mathcal{X}$
or $[a,B]\cap \mathcal{X} = \emptyset$. A set
$\mathcal{X}\subseteq \mathcal{R}$ is \textbf{Ramsey null} if for
every neighborhood $[a,A]$ there exists  $B\in [a,A]$ such that
$[a,B]\cap \mathcal{X} = \emptyset$.
\end{defn}

\begin{defn}
We say that $(\mathcal{R}, \leq,
r)$ is a \textbf{topological Ramsey space} if subsets of
$\mathcal{R}$ with the Baire property are Ramsey and meager
subsets of $\mathcal{R}$ are Ramsey null.
\end{defn}

\medskip

(A.4)(\textit{Finitization}) There is a quasi order $\leq_{fin}$ on
$\mathcal{AR}$ such that:
    \begin{itemize}
    \item[{(i)}]$A\leq B$ iff
    $(\forall n)\ (\exists m) \ \ (r_{n}(A)\leq_{fin} r_{m}(B))$.
    \item[{(ii)}]$\{b\in \mathcal{AR} : b\leq_{fin} a\}$ is finite, for every
    $a\in \mathcal{AR}$.
    \end{itemize}

\vspace{.25 cm}

 Given $A\in\mathcal{R}$ and $a\in\mathcal{AR}(A)$,
we define the \textit{depth of} $a$ \textit{in}
$A$ as $$depth_{A}(a): = min\{n :a\leq_{fin}r_n(A)\}.$$

\begin{lemma}\label{LenthDepth}
Given $A\in\mathcal{R}$ and $a\in\mathcal{AR}(A)$, $|a|\leq depth_{A}(a)$.
\end{lemma}
\qed

\begin{itemize}

\item[{(A.5)}](\textit{Amalgamation}) Given $a$ and $A$
with $depth_{A}(a)=n$, the following holds:
\begin{itemize}
\item[{(i)}] $(\forall B\in [n,A])\ \ ([a,B]\neq\emptyset)$.
\item[{(ii)}] $(\forall B\in [a,A])\ \ (\exists A'\in [n,A])\ \
([a,A']\subseteq [a,B])$.
\end{itemize}
\item[{(A.6)}](\textit{Pigeon Hole Principle}) Given $a$
and $A$ with $depth_{A}(a) = n$, for every $\mathcal{O}\subseteq\mathcal{AR}_{|a|+1}$ there is $B\in
[n,A]$ such that $r_{|a|+1}[a,B]\subseteq\mathcal{O}$ or $r_{|a|+1}[a,B]\subseteq\mathcal{O}^c$.
\end{itemize}

\vspace{0.4cm}

{\bf Abstract Ellentuck theorem:}

\begin{thm}[Carlson]\label{AbsEll}
Any $(\mathcal{R}, \leq, r)$
with $\mathcal{R}$ metrically closed and satisfying (A.1)-(A.6) is a
Ramsey space.
\end{thm}
\qed

\section{Abstract versions}\label{AbstactVersions}

 The following is the main result of this paper. As announced in the introduction, we are going to avoid the \textit{indirect approach} in the proof; that is, we will not make use of the abstract Ellentuck theorem.

\begin{thm}[Abstract version of Galvin's lemma.]\label{abstractGalvin}
Given $(\mathcal{R}, \leq, r)$ with $\mathcal{R}$ metrically closed and satisfying (A.1)-(A.6), ${\cal F}\subseteq{\cal AR}$,  and $A\in{\cal R}$, there exists $B\leq A$ such that one of the following holds:
\begin{enumerate}
\item ${\cal AR}(B)\cap{\cal F}=\emptyset$, or
\item $(\forall C\leq B)$ $(\exists \ n\in \mathbb{N})$ $(r_n(C)\in{\cal F})$.
\end{enumerate}
\end{thm}

\begin{proof}
Fix ${\cal
F}\subseteq {\cal AR}$. Given $A\in {\cal R}$ and $a\in {\cal AR}$, we say that $A$
\emph{\textbf{accepts}} $a$ if for every $B\in [a,A]$ there exists $n\in \mathbb{N}$
such that $r_n(B)\in{\cal F}$. We say that $A$
\emph{\textbf{rejects}} $a$ if $[a,A]\neq\emptyset$ and no element
of $[depth_A(a),A]$ accepts $a$; and we say that $A$
\emph{\textbf{decides}} $a$ if $A$ either accepts or rejects $a$.
This \textit{combinatorial forcing} has the following properties:

\begin{clm}\label{forcingFacts}
\begin{enumerate}
\item If $A$ accepts $a$, then every $B\leq A$ accepts $a$.
\item If $A$ rejects $a$, then every $B\leq A$ rejects $a$, if $[a,B]\neq\emptyset$.
\item For every $A\in {\cal R}$ and every $a\in {\cal AR}(A)$ there exists $B\in [depth_A(a),A]$ which decides $a$.
\item If $A$ accepts $a$ then $A$ accepts every $b\in r_{|a|+1}[a,A]$.
\item If $A$ rejects $a$ then there exists $B\in [depth_A(a),A]$ such that $A$ does not accept any $b\in r_{|a|+1}[a,B]$.
\end{enumerate}
\end{clm}
\begin{proof}Parts 1, 2, 3 and 4 follow from the definitions. To prove 5, let $\mathcal{O} = \{b\in\mathcal{AR}_{|a|+1} : A\ \mbox{accepts}\ b\}$. By A6, there exists $B\in [depth_A(a),A]$ such that $r_{|a|+1}[a,B]\subseteq\mathcal{O}$ or $r_{|a|+1}[a,B]\subseteq\mathcal{O}^c$. The first alternative is not possible since $A$ rejects $a$. Then the second alernative holds and hence $B$ is as required.
\end{proof}

\begin{clm}\label{decidesClaim}
 Given $A\in{\cal R}$, there exists $B\leq A$ which decides every $b\in {\cal AR}(B)$.
\end{clm}
\begin{proof}
Notice that for every $B\in\mathcal{R}$ and every $k\in\mathbb{N}$ the set $\{b\in\mathcal{AR}(B) : depth_{B}(b)=k\}$ is finite, by A4. Using this fact and part 3 of Claim \ref{forcingFacts} iteratively, we can build a sequence $(B_n)_{n\in\mathbb{N}}\subseteq\mathcal{R}$ such that:

\begin{enumerate}
\item $B_0=A$.
\item $(\forall n>0)\ (B_n\in [n-1,B_{n-1}])$
\item $(\forall n>0)$\ ($B_n$ decides every $b\in\mathcal{AR}(B_n)$ with $depth_{B_n}(b)=n-1$).
\end{enumerate}

 Notice that $\bigcap_n [n,B_n]\neq\emptyset$, since $\mathcal{R}$ is metrically closed. If we take $B\in\bigcap_n [n,B_n]$ then $B$ is as required.
\end{proof}

 Continuing with the proof of Theorem \ref{abstractGalvin}, given $A\in\mathcal{R}$, fix $\hat{B}\leq A$ which decides every $b\in {\cal AR}(\hat{B})$. If $\hat{B}$ accepts $\emptyset$ then part 2 of Theorem \ref{abstractGalvin} holds. Otherwise,  we build a sequence $(C_n)_{n\in\mathbb{N}}\subseteq\mathcal{R}$ such that:

\begin{enumerate}
\item $C_0=\hat{B}$.
\item $(\forall n>0)\ (C_n\in [n-1,C_{n-1}])$
\item $(\forall n)$\ ($C_n$ rejects every $b\in\mathcal{AR}(C_n)$ with $|b|\leq n$).
\end{enumerate}

 So let $C_0=\hat{B}$. Then, $C_1$ is obtained applying part 5 of Claim \ref{forcingFacts}, since $C_0$ rejects $\emptyset$ and decides any other $b\in \mathcal{AR}(C_0)$.

\medskip

 Suppose we have define $C_n$ rejecting every $b\in\mathcal{AR}(C_n)$ with $|b|\leq n$. Again, applying part 5 of Claim \ref{forcingFacts} iteratively (and also applying Lemma \ref{LenthDepth}), for every $k\geq 0$ define $C_n^{k}$ such that:

\begin{enumerate}
\item[{(a)}] $C_n^0 \in [n,C_n]$.
\item[{(b)}] $(\forall k>0)\ C_n^k \in [n+k,C_n^{k-1}]$.
\item[{(c)}] $C_n^k$ rejects every $b \in \mathcal{AR}(C_n^k)$ with $|b| = n+1$ and $depth_{C_n^k}(r_n(b)) = n+k$.
\end{enumerate}

 Here, $r_n(b)$ is that unique $a$ such that $|a| = n$ and
$a$ is an initial segment of $b$. That is, if $b = r_{n+1}(A)$ for some $A$ then $a = r_n(A)$. It
is unique because of axioms A1-A3.

\medskip

 Take $C_{n+1}\in\bigcap_k [n+k,C_n^k]$. Then $C_{n+1}\in [n,C_n]$ and $C_{n+1}$ rejects every $b\in\mathcal{AR}(C_{n+1})$ with $|b|\leq n+1$.

\medskip

 This completes the definition of the $C_n$'s..

\medskip

 Now, take $B\in\bigcap_n [n,C_n]$. If $b \in \mathcal{AR}(B)$, by the choice of
$B$, there exists $n \geq |b|$ such that $[b,C_n]\neq\emptyset$ (that
is, $b \in \mathcal{AR}(C_n)$). Then $C_n$ rejects b, by Condition 3, and therefore so does
$B$. Hence Part 1 of Theorem \ref{abstractGalvin} holds with $B$ as witness. This completes the proof.

\end{proof}

 Using Theorem \ref{abstractGalvin}, we give a simpler proof of Theorem 1.7 of \cite{mij}, which is an abstract version of Ramsey's theorem.

\begin{rmrk}
Throughout the rest of this paper we will identify any element of $\mathbb{N}$ with the set of its predecessors.
\end{rmrk}

\begin{thm}[Abstract version of Ramsey's theorem]\label{AbsRam}
Given $(\mathcal{R}, \leq, r)$ with $\mathcal{R}$ metrically closed and satisfying (A.1)-(A.6), the following holds. Let $k, s\in\mathbb{N}$ and $A\in\mathcal{R}$ be given. Then, for every coloring $c : \mathcal{AR}_k \rightarrow s$, there exists $B\leq A$ such that $c$ is constant in $\mathcal{AR}_k(B)$.
\end{thm}
\begin{proof}
Fix $k, s\in\mathbb{N}$ and $A\in\mathcal{R}$. Without a loss of generality, we can assume $s=2$. Then the result follows from Theorem \ref{abstractGalvin} applied to $\mathcal{F} = c^{-1}(\{0\})$ and $A$.
\end{proof}

 \textbf{Notation:} For $k,m\in\mathbb{N}$, $A\in\mathcal{R}$ and $b\in\mathcal{AR}(A)$, let us define $\mathcal{AR}_k^m(A) : = \{a\in\mathcal{AR}_k(A) : depth_A(a)=m\}$, $\mathcal{AR}_k^m(A,b) : = \{a\in\mathcal{AR}_k^m(A) : a\leq_{fin}b\}$.

\medskip

 With this notation, we state and prove the following abstract version of finite Ramsey's theorem. In \cite{carsimp}, a similar result is presented but the proof given in \cite{carsimp} uses the abstract Ellentuck theorem.

\begin{thm}[Abstract version of finite Ramsey's theorem]\label{AbsFiniteRamsey}
Given $(\mathcal{R}, \leq, r)$ with $\mathcal{R}$ metrically closed and satisfying (A.1)-(A.6), the following holds. Let $k, n, s\in\mathbb{N}$ and $A\in\mathcal{R}$ be given. Then, there exists $m\in\mathbb{N}$ such that for every coloring $c : \mathcal{AR}_k^m(A)\rightarrow s$, there exists $b\in\mathcal{AR}_n^m(A)$ such that $c$ is constant in $\mathcal{AR}_k^m(A,b)$.
\end{thm}
\begin{proof}
Fix $k, n, s\in\mathbb{N}$ and $A\in\mathcal{R}$ such that for all $m$ there exists $c_m$ witnessing that the thesis of the theorem fails for $m$. Let us define $c : \mathcal{AR}_k \rightarrow s$ as: $$c(a) = c_{d(a)}(a)$$where $d(a) = depth_A(a)$, for all $a\in\mathcal{AR}_k(A)$; and $c(a)=0$ if $a\not\in\mathcal{AR}_k(A)$. By Theorem \ref{AbsRam}, there exists $B\leq A$ such that $c$ is constant in $\mathcal{AR}_k(B)$. Now, choose any $b\in\mathcal{AR}_n(B)$ and let $\hat{m} = depth_A(b)$. Notice that the following holds:

\begin{enumerate}
\item $b\in\mathcal{AR}_n^{\hat{m}}(A)$, and
\item $\mathcal{AR}_k^{\hat{m}}(A,b)\subset\mathcal{AR}_k(B)$.
\end{enumerate}

 (To prove 2, notice that if $a\leq_{fin}b$ and $b\in\mathcal{AR}(B)$ then $depth_B(a)\geq 0$. Hence, $a\in\mathcal{AR}(B)$ by A5(i).) Therefore, $c$ is constant in $\mathcal{AR}_k^{\hat{m}}(A,b)$. But this contradicts the fact that $c\upharpoonright\mathcal{AR}_k^{\hat{m}}(A,b) = c_{\hat{m}}$. This completes the proof.
\end{proof}

 Finally, we present the following consequence of theorem \ref{abstractGalvin}:

\begin{coro}\label{abiertos}
(Carlson) The metric Borel subsets of ${\cal R}$ are Ramsey.
\end{coro}

\begin{proof}
We only need to prove the result for metric open sets, because the Ramsey property is preserved by countable unions and complementation. Let $\mathcal{X}$ be a metric open subset of $\mathcal{R}$ and fix a nonempty $[a,A]$. Without a loss of generality we can assume $a = \emptyset$. Since $\mathcal{X}$ is open, there exists $\mathcal{F}\subseteq {\cal AR}$ such that ${\cal X} = \bigcup_{b\in {\cal F}} [b]$. Let $B\leq A$ be as in Theorem \ref{abstractGalvin}. If Part 1 of the theorem holds then $[0,B]\subseteq\mathcal{X}^c$ and if Part 2 holds then $[0,B]\subseteq\mathcal{X}$.
\end{proof}

\section{Some instances}

\subsection{Classical versions}

 If $(\mathcal{R}, \leq, r)$ is Ellentuck's space, that is, $\mathcal{R} := \mathbb{N}^{[\infty]}$,  $\leq\ :=\ \subseteq$ and $r(n,A) :=$ \textit{the first} $n$ \textit{elements of} $A$, then classical Galvin's lemma, Ramsey's theorem and the Galvin-Prikry theorem \cite{galpri} are easily obtained from Theorems \ref{abstractGalvin} and  \ref{AbsRam}, and  Corollary \ref{abiertos}, respectively. For every $X\subseteq\mathbb{N}$, let $X^{[k]} = \{Y\subseteq X = |Y| = k\}$. Then, in this case $\mathcal{AR}_k = \mathcal{AR}_k(\mathbb{N}) = \mathbb{N}^{[k]} = \{X\subseteq\mathbb{N} = |X| = k\}$ and $\mathcal{AR} = \mathcal{AR}(\mathbb{N}) = \mathbb{N}^{[<\infty]}$. Finite Ramsey's theorem is also obtained from Theorem \ref{AbsFiniteRamsey} but the proof needs some more work:

\begin{coro}[Finite Ramsey's theorem]
 Let $k, n, s\in\mathbb{N}$ be given. Then, there exists $M\in\mathbb{N}$ such that for every partition $c : M^{[k]}\rightarrow s$, there exists $H\in M^{[n]}$ such that $c$ is constant on $H^{[k]}$.
\end{coro}
\begin{proof}
 Given $k,n, s\in\mathbb{N}$, let us apply Theorem \ref{AbsFiniteRamsey} to $k +1,n+1,r$ and $A = \mathbb{N}$, for $(\mathcal{R}, \leq, r)$ equal to Ellentuck's space.  First, notice that in this case the following holds for any $i,j\in\mathbb{N}$:

\begin{enumerate}
\item $\mathcal{AR}_i = \mathcal{AR}_i(\mathbb{N}) = \mathbb{N}^{[i]}$
\item $\mathcal{AR}_i^j(\mathbb{N}) = \{x\in j^{[i]} : j-1\in x\}$
\item $\mathcal{AR}_i^j(\mathbb{N},b) = \{x\subseteq b :  max (x) = max(b) = j-1\}$, for any $b\in\mathcal{AR}_i^j(\mathbb{N})$.
\end{enumerate}

 Let $m>1$ be as in Theorem \ref{AbsFiniteRamsey} applied to $k +1,n+1,r$ and $A = \mathbb{N}$. Now, consider a coloring $$c : (m-1)^{[k]}\rightarrow s$$and define $$\hat{c} : \{x\in m^{[k+1]} : m-1\in x\}\rightarrow s$$ as $$\hat{c}(x) = c(x\setminus\{m-1\}).$$ By the choice of $m$, there exists $\hat{H}\in m^{[n+1]}$ with $max(\hat{H})=m-1$ such that $\hat{c}$ is constant in $\{x\in\hat{H}^{[k+1]} : max(x) = m-1\}$. Let $H = \hat{H}\setminus\{m-1\}$. Notice that $H\in(m-1)^{[n]}$ and $c$ is constant in $H^{[k]}$. So $M = m-1$ is as required. This completes the proof.
\end{proof}

\subsection{Vector versions}

\subparagraph*{Matrices.}
 Let $F$ be a finite field. An $\mathbb{N}\times\mathbb{N}$--matrix over $F$ is a mapping $A:\mathbb{N}\times\mathbb{N}\rightarrow F$. Let $\mathcal{M}_{\infty}(F)$ denote the collection of all row-reduced echelon $\mathbb{N}\times\mathbb{N}$--matrices over $F$. For $A,B\in\mathcal{M}_{\infty}(F)$ write $A\leq B$ if and only if each row of $A$ is in the closed linear subspace of $F^{\mathbb{N}}$ generated by the rows of $B$.

\medskip

 For $A\in\mathcal{M}_{\infty}(F)$ and $n\in\mathbb{N}$, let $p_n(A) : = min\{j : A_n(j) \neq 0\}$. We define now the approximation function $r$ on $\mathbb{N}\times\mathcal{M}_{\infty}(F)$ as: $$r(0,A) = r_0(A): = \emptyset$$and $$r(n,A) = r_n(A) := A\upharpoonright((n+1)\times p_n(A)).$$ for $n>0$. In \cite{todo}, it is shown that $(\mathcal{M}_{\infty}(F), \leq, r)$ satisfies (A1)-(A6). So we are going to apply the results of Section \ref{AbstactVersions} to obtain the corresponding versions of Ramsey's theorem and Galvin's lemma within this context.

\medskip

 For $n,m\in\mathbb{N}$, let $\mathcal{M}_{n\times m}(F)$ denote the collection of all row-reduced echelon $n\times m$--matrices over $F$, and let $\mathcal{M}_{<\infty}(F) = \bigcup_{n,m\in\mathbb{N}}\mathcal{M}_{n\times m}(F)$, the collection of all row-reduced echelon matrices over $F$ with a finite number of rows and columns. In this context, $$\mathcal{AR}_n = \bigcup_{m\in\mathbb{N}}\mathcal{M}_{n\times m}(F),$$ for every $n\in\mathbb{N}$; and $$\mathcal{AR} = \mathcal{M}_{<\infty}(F).$$Now, for $A\in\mathcal{M}_{\infty}(F)$ and $a\in\mathcal{M}_{<\infty}(F)$, write $a\sqsubset A$ if there exists $n$ such that $a = r_n(A)$; also let $\mathcal{M}_{<\infty}^A(F)$ denote the set $\{a\in\mathcal{M}_{<\infty}(F) : \exists B\leq A\ (a\sqsubset B)\}$. Analogously define $\mathcal{M}_{n\times m}^A(F)$, for every $n,m\in\mathbb{N}$. So, in this case we have $$\mathcal{AR}(A) = \mathcal{M}_{<\infty}^A(F)$$and $$\mathcal{AR}_n(A) = \bigcup_{m\in\mathbb{N}}\mathcal{M}_{n\times m}^A(F),$$for every $n\in\mathbb{N}$. With this notation, in virtue of the results of Section \ref{AbstactVersions}, we can state versions of Galvin's lemma and Ramsey's theorem for matrices:

\begin{coro}\label{matrixGalvin}
(Galvin's lemma for matrices) For every $\mathcal{F}\subseteq\mathcal{M}_{<\infty}(F)$ and $A\in\mathcal{M}_{\infty}(F)$, there exists $B\leq A$ such that one of the following holds:
\begin{enumerate}
\item $\mathcal{M}_{<\infty}^B(F)\cap\mathcal{F} = \emptyset$, or
\item For every $C\leq B$ there exists $a\in\mathcal{F}$ such that $a\sqsubset C$.
\end{enumerate}
\end{coro}
\qed
\medskip

\begin{coro}\label{matrixRamsey}
(Ramsey's theorem for matrices) Let $n,s\in\mathbb{N}$ and $A\in\mathcal{M}_{\infty}(F)$ be given. For every finite coloring $c: \bigcup_{m\in\mathbb{N}}\mathcal{M}_{n\times m}(F) \rightarrow s$, there exists $B\leq A$ such that $\bigcup_{m\in\mathbb{N}}\mathcal{M}_{n\times m}^B(F)$ is monochromatic.
\end{coro}
\qed

 Now, given $k,n,m\in\mathbb{N}$ and $a\in\mathcal{M}_{n\times m}(F)$, let $\mathcal{M}_{k\times m}^a(F)$ denote the collection of all $k\times m$-matrices $b$ such that every row of $b$ is in the linear span generated by the rows of $a$ in $F^m$. From Theorem \ref{AbsFiniteRamsey} we obtain the following version of finite Ramsey's theorem for matrices:

\begin{coro}\label{FiniteMatrixRamsey}
(Finite Ramsey's theorem for matrices) Given $k,n,s\in\mathbb{N}$ there exists $m$ such that for every coloring $c : \mathcal{M}_{k\times m}(F) \rightarrow s$ there exists $a\in\mathcal{M}_{n\times m}(F)$ such that $\mathcal{M}_{k\times m}^a(F)$  is monochromatic.
\end{coro}
\qed

 Next, the instance of Corollary \ref{abiertos} in this context:

\begin{coro}\label{MatrixOpen}
 Every  metric Borel subset of $\mathcal{M}_{\infty}(F)$ is Ramsey.
\end{coro}

\subparagraph*{Vector spaces.}
 Now, we will obtain vector versions of Ramsey's theorem and Galvin's lemma from Corollary \ref{matrixGalvin}. Also, Graham-Leeb-Rothschild theorem \cite{graleeroth} is obtained from Corollary \ref{FiniteMatrixRamsey}; and an infinitary version of it due to Carlson (\cite{carlson}), which is a vector version of Galvin-Prikry's theorem \cite{galpri}, is also obtained from Corollary \ref{MatrixOpen}. Some definitions are needed:

\medskip

 Given a finite field $F$, let

\begin{enumerate}
\item[{}] $\mathcal{V}_{\infty}(F) : =$ the set of infinite-dimensional closed subspaces of $F^{\mathbb{N}}$.
\item[{}] $\mathcal{V}_n^m(F) : =$ the set of $n$-dimensional subspaces of $F^m$, for every $n,m\in\mathbb{N}$ with $n\leq m$.
\item[{}] $\mathcal{V}_n^{<\infty}(F) : = \bigcup_m \mathcal{V}_n^m(F)$, for every $n\in\mathbb{N}$.
\item[{}] $\mathcal{V}^{<\infty}(F) : = \bigcup_n \mathcal{V}_n^{<\infty}(F)$.
\end{enumerate}

\begin{defn}
Given $V\in\mathcal{V}_{\infty}(F)$ and  $W\in\mathcal{V}^{<\infty}(F)$, we say that  $W$ is an \textbf{initial segment} of $V$, and write $W\sqsubset V$, if there exist $a\in\mathcal{M}_{<\infty}(F)$ and $B\in\mathcal{M}_{\infty}(F)$ such that the rows of $a$ form a basis for $W$, the closed linear span of the rows of $B$ is $V$ and $a$ is an approximation (in the sense of $(\mathcal{M}_{\infty}(F), \leq, r)$) of $B$.
\end{defn}

\medskip

 Fix $V\in\mathcal{V}_{\infty}(F)$. Let $$\mathcal{V}_{\infty}(F,V) : = \{V'\in\mathcal{V}_{\infty}(F) : V'\ \mbox{is a subspace of}\ V\}.$$ and for $n,m\in\mathbb{N}$ with $n\leq m$, let $$\mathcal{V}_n^m(F,V) : = \{W\in\mathcal{V}_n^m(F) : \exists V'\in\mathcal{V}_{\infty}(F,V)\ \ (W\sqsubset V')\}.$$ Also, let $$\mathcal{V}_n^{<\infty}(F,V) : = \bigcup_{m\geq n} \mathcal{V}_n^m(F,V)$$ and $$\mathcal{V}^{<\infty}(F,V) : = \bigcup_n \mathcal{V}_n^{<\infty}(F,V).$$

 From the results above we obtain the following:

\begin{coro}[Vector Galvin's lemma]\label{VectorGalvin}
For every $\mathcal{F}\subseteq\mathcal{V}^{<\infty}(F)$ there exists $V\in\mathcal{V}_{\infty}(F)$ such that one of the following holds:

\begin{enumerate}
\item  $\mathcal{V}^{<\infty}(F,V)\cap\mathcal{F} = \emptyset$, or
\item For every infinite-dimensional subspace $V'$ of $V$ there exists $W\in\mathcal{F}$ such that $W\sqsubset V'$.
\end{enumerate}
\end{coro}
\begin{proof}
Let $\hat{\mathcal{F}} = \{a\in\mathcal{M}_{<\infty}(F): \exists W\in\mathcal{F}\ (\mbox{the rows of}\ a\ \mbox{form a basis for}\ W)\}$, and fix  $B\in\mathcal{M}_{\infty}(F)$ satisfying the conclusion of Corollary \ref{matrixGalvin} for $\hat{\mathcal{F}}$. Let $V$ be the closed linear span generated by the rows of $B$. If $W\in\mathcal{V}^{<\infty}(F,V)$ and $a$ is such that its rows form a basis for $W$ then $a\in\mathcal{M}_{<\infty}^B(F)$. So, if Part 1 of Corollary \ref{matrixGalvin} is true then $\mathcal{V}^{<\infty}(F,V)\cap\mathcal{F} = \emptyset$. On the other hand, if Part 2 of Corollary \ref{matrixGalvin} holds and $V'\in\mathcal{V}_{\infty}(F,V)$, let $B'\leq B$ be such that the closed linear span of the rows of $B'$ is $V'$. Then, there exists $a\in\hat{\mathcal{F}}$ such that $a\sqsubset B'$. Let $W$ be the linear space generated by the rows of $a$. Then $W\in\mathcal{F}$ and $W\sqsubset V'$.
\end{proof}

 The following is a direct consequence of Corollary \ref{VectorGalvin}:

\begin{coro}[Vector Ramsey's theorem]\label{VectorRamsey}
For every $n,s\in\mathbb{N}$ and every coloring $c : \mathcal{V}_n^{<\infty}(F) \rightarrow s$ there exists $V\in\mathcal{V}_{\infty}(F) $ such that $c$ is constant in $\mathcal{V}_n^{<\infty}(F,V)$.
\end{coro}
\qed

 Now, the Graham-Leeb-Rothschild theorem is obtained directly from Corollary \ref{VectorRamsey} (or from Corollary \ref{FiniteMatrixRamsey}):

\begin{coro}[Graham-Leeb-Rothschild theorem \cite{graleeroth}]\label{G-L-R}
For every $k,n,s\in\mathbb{N}$, there exists $m\in\mathbb{N}$ large enough so that for every partiton of the $k$-dimensional subspaces of $F^m$ into $s$ classes there exists an $n$-dimensional subspace $V$ of $F^m$ such that the collection of $k$-dimensional subspaces of $V$ lies in one only class.
\end{coro}
\qed

 We conclude this section with a proof of the infinitary version of the Graham-Leeb-Rothschild theorem due to Carlson. It is a vector version of the Galvin-Prikry theorem:

\begin{coro}[Carlson \cite{carlson}]\label{VectorGalvinPrikry}
 If $\mathcal{X}\subseteq\mathcal{V}_{\infty}(F)$ is Borel then there exists $V\in\mathcal{V}_{\infty}(F)$ such that either all closed infinite subspace of $V$ is in $\mathcal{X}$ or all closed infinite subspaces of $V$ are in the complement of $\mathcal{X}$.
\end{coro}
\begin{proof}
Every open subset of $\mathcal{V}_{\infty}(F)$ can be easily identified with an open subset of $\mathcal{M}_{\infty}(F)$, with the product topology inherited from $F^{\mathbb{N}\times\mathbb{N}}$, regarding $F$ as a discrete space. This correspondence is actually an homeomorphism. So, the result holds by Corollary \ref{MatrixOpen}.
\end{proof}

\subsection{Dual versions}

Let $(\omega)^{\omega}$ be the set of all the infinite partitions $X = (X_i)_{i\in\mathbb{N}}$ of $\mathbb{N}$ such that $$i<j \rightarrow min (X_i) < min (X_j).$$ Given $X,Y\in (\omega)^{\omega}$, we say that $X$ is \textit{coarser} than $Y$ if very block in $Y$ is a subset of some block in $X$. Pre-order $(\omega)^{\omega}$ as follows:

$$X \leq Y \longleftrightarrow \ X\ \mbox{is coarser than}\ Y$$

 For every $k,n\in\mathbb{N}$ let $(n)^k$ be the set of all the $k$-partitions of $n$, i.e., partitions of $n$ into $k$ pieces. Also, for every $k\in\mathbb{N}$, let $(<\omega)^k := \bigcup_{n\in\mathbb{N}}(n)^k =$ the set of all the $k$-partitions of some integer. Finally, set $(<\omega)^{<\omega} = \bigcup_{k\in\mathbb{N}}(<\omega)^k$.

\medskip

 Let us define $r : \mathbb{N}\times (\omega)^{\omega} \rightarrow (<\omega)^{<\omega} $ in the following way:

$$\forall n\ \forall X = (X_i)_{i\in\mathbb{N}},\ r(n,X) = r_n(X) = (X_i\cap min (X_n) \})_{i<n}\setminus\{\emptyset\}.$$

\medskip

 It is known that $((\omega)^{\omega},\leq, r)$ satisfies (A.1)-(A.6) and is a closed subset of the product space $((<\omega)^{<\omega})^{\mathbb{N}}$, regarding $(<\omega)^{<\omega}$ as a discrete space (see \cite{todo}). So, we can state the corresponding versions of Theorems \ref{abstractGalvin}, \ref{AbsRam} and \ref{AbsFiniteRamsey}. For $s\in (<\omega)^{<\omega}$ and $X\in(\omega)^{\omega}$, write $s\sqsubset X$ if $(\exists n)(s = r_n(X))$.

\begin{coro}\label{PartitionGalvin}(Dualization of Galvin's lemma.)
 Given ${\cal F}\subseteq (<\omega)^{<\omega}$ and $X\in(\omega)^{\omega}$ there exists $Y\in(\omega)^{\omega}$ such that one of the following holds:
\begin{enumerate}
\item $(<\omega,Y)^{<\omega}\cap{\cal F}=\emptyset$, or
\item $\forall Z\in (Y)^{\omega} (\exists s\in {\cal F})(s\sqsubset Z)$.
\end{enumerate}
\end{coro}
\qed

\begin{coro}[Dualization of Ramsey's theorem; Halbeisen \cite{Halb}]\label{DualRamsey}
For all $k, s\in\mathbb{N}$ and every coloring $c : (<\omega)^k \rightarrow s$ there exists $Y\in (\omega)^{\omega}$ such that $(<\omega, Y)^k$ is monochromatic.
\end{coro}
\qed

 Interestingly, the proof given in \cite{Halb} of Corollary \ref{DualRamsey} uses the Dual Ramsey theorem of Carlson and Simpson \cite{carsimp2}. Notice that our proof of it is simpler. The dualization of the finite Ramsey theorem, (namely, Ramsey's theorem for $n$-parameter sets) can be easily obtained from Corollary \ref{DualRamsey} by a typical compactness argument.

\begin{coro}
(Ramsey's theorem for $n$-parameter sets; Graham-Rothschild \cite{graroth}) For all positive integers $r$ and $k\leq m$ there exists $n\in\mathbb{N}$ large enough for the following to hold. For every coloring $c : (n)^k \rightarrow s$ there exists $t\in (n)^m$ such that $c$ is constant in $(t)^k$.
\end{coro}
\begin{proof}
 Fix  positive integers $r$ and $k\leq m$. Suppose the conclusion fails, and for every $n\in\mathbb{N}$ choose $c_n$, an $r$-coloring of $(n)^k$ witnessing this fact. For every $t\in (<\omega)^k$, use the notation $\#(t)$ to denote the unique $n\in\mathbb{N}$ such that $t$ is a $k$-partition of $n$. Let us define  $c : (<\omega)^k \rightarrow s$ as follows: $$\forall t\in (<\omega)^k,\ c(t) = c_{\#(t)}(t)$$ By Corollary \ref{DualRamsey}, there exists $Y\in (\omega)^{\omega}$ such that $(<\omega, Y)^k$ is monochromatic for $c$. Choose any $t\in(<\omega, Y)^m$ and let $n = \#(t)$. Then $t\in (n)^m$ and $(t)^k\subset(<\omega, Y)^k$. So $c$ is constant in $(t)^k$, but $c = c_n$ in $(t)^k$. A contradiction.
\end{proof}

\begin{rmrk}
Ramsey's theorem \cite{ramsey} is also a consequence of corollary \ref{DualRamsey}: for every finite coloring  $c$ of $\mathbb{N}^{[k]}$, define a finite coloring $d$ of $(<\omega)^{k+1}$ in this way: $d(s) = c(\{min\ x: x\ \mbox{is a block of}\ s\}\setminus\{0\})$.
\end{rmrk}

We conclude this section with one more direct consequence of Corollary \ref{PartitionGalvin}:

\begin{coro}
[Dual Galvin-Prikry theorem; Carlson and Simpson \cite{carsimp2}] Given a partition $(\omega)^{\omega} = C_0\cup C_1\dots \cup C_{r-1}$ where each $C_i$ is Borel, there exists $X\in (\omega)^{\omega}$ such that $(X)^{\omega}\subseteq C_i$ for some $i$.
\end{coro}
\qed

\section{Final comments}

The importance of Theorems \ref{abstractGalvin}, \ref{AbsRam} and \ref{AbsFiniteRamsey} is partially in the variety of instances which follow as special cases. As we have seen, some of them are well known important results like Galvin's lemma, Ramsey's theorem or the Graham-Leeb-Rothschild theorem. Nevertheless, some of them have been little explored before, as far as we are concerned. For example, this is the case of Corollary \ref{PartitionGalvin}, the dualization of Galvin's lemma. And it is also the case of the version of Galvin's lemma obtained from Theorem \ref{abstractGalvin} when we consider the space $FIN_k^{[\infty]}$ of all the infinite block sequences of elements of $FIN_k$, the discretization  of the positive part of the unit sphere of the Banach space $c_0$ used by Gowers to study a sort of stability for Lipschitz functions (please see \cite{Gow} and \cite{todo} for the definitions). We know from \cite{todo} that $FIN_k^{[\infty]}$ is a topological Ramsey space. So in virtue of Theorem \ref{abstractGalvin} we can prove directly that every (metric) Borel subset of $FIN_k^{[\infty]}$ is Ramsey.

\medskip

Finally, we would like to conclude by mentioning the following. In the proof of Theorem \ref{abstractGalvin}, a technique of \textit{selection by diagonalization (or by fusion)} is used recurrently; see for example the proof of Claim \ref{decidesClaim}. We can now attempt to isolate from it a notion of  \textit{abstract selective coideal} analog to the concept of selective coideal on $\mathbb{N}$ (see \cite{mathias}) to generalize the results contained in \cite{mij}, where a notion \textit{selective ultrafilter} corresponding to topological Ramsey spaces is given. This in turn could lead us to an abstract approach to local Ramsey theory. This was in part the motivation for this paper.

\end{document}